\documentclass[12pt]{amsart}

\usepackage[utf8]{inputenc}
\usepackage{color}

\newcommand{\norm}[2]{\left\|#1\right\|_{#2}}
\newcommand{\scalar}[2]{\langle{ #1},{#2} \rangle}

\newcommand{\set}[1]{\left\{#1\right\}}

\newcommand{\lr}[1]{\left(#1\right)}
\newcommand{\nat}{\mathbb N}
\newcommand{\real}{\mathbb R}
\newcommand{\rk}{\operatorname{rank}}
\newcommand{\domain}{\mathcal D}
\newcommand{\range}{\mathcal R}
\newcommand{\yd}{y^{\delta}}
\newcommand{\kast}{k_{\ast}}
\newcommand{\xkd}[1]{x_{#1}^\delta}
\newcommand{\cd}{c\lr{\frac 1 d}}
\newcommand{\msp}{\mu}
\newtheorem{thm}{Theorem}

\newtheorem{lemma}{Lemma}
\newtheorem{prop}{Proposition}
\newtheorem{cor}{Corollary}

\theoremstyle{definition}
\newtheorem{de}{Definition}
\theoremstyle{remark}
\newtheorem{rem}{Remark}
\newtheorem{xmpl}{Example}
\renewenvironment{description}[1][0pt]
  {\list{}{\labelwidth=0pt \leftmargin=#1
   }}
  {\endlist}

\newcommand{\pmathe}[1]{#1}

\author{Peter Math\'e}
\address{Weierstra{\ss} Institute for Applied Analysis and Stochastics, Mohrenstra{\ss}e 39, 10117 Berlin,  Germany}

\author{Bernd Hofmann}
\address{Faculty of Mathematics, Chemnitz University of Technology,
 09107 Chemnitz,  Germany}

\title[Tractability of ill-posed problems]{Tractability of linear ill-posed problems in Hilbert space}
\date{\today}
\begin{document}
\begin{abstract}
  We introduce a notion of tractability for ill-posed operator
  equations in Hilbert space. For such operator equations the
  asymptotics of the best possible rate of reconstruction in terms of
  the underlying noise level is known in many cases. However, the relevant
  question is, which level of discretization, again driven by the
  noise level,  is required in order to
  achieve this best possible accuracy.
  The proposed concept adapts the one from Information-based Complexity.
  Several examples indicate the relevance of this concept in the light of the curse of dimensionality.
 \end{abstract}
\maketitle
\section{Introduction}
\label{sec:intro}

We shall introduce the concept of \emph{tractability}
for linear \emph{ill-posed problems}, which are modeled in the form of
\emph{operator equations}
\begin{equation} \label{eq:opeq}
 A\,x\,=\,y\,,
\end{equation}
where some injective bounded linear operator~$A\colon X \to Y$ is
acting between real infinite dimensional Hilbert spaces~$X$
and~$Y$.
We consider the {noise model}
\begin{equation}
  \label{eq:noise}
  \yd:= A x + \delta\xi,
\end{equation}
where the unknown noise element~$\xi$ is norm bounded by one, such
that~$\norm{A x - \yd}{Y}\leq \delta$. \pmathe{The goal is to
  approximately reconstruct the unknown element~$x$ from the noisy data~$\yd$.}
There is vast literature available referring to the analysis of such ill-posed problems, and the standard monograph is~\cite{MR1408680}. The optimal
rates of reconstruction, under appropriate smoothness assumptions,  are known in many cases,
and often  these are given in terms of the {noise level}~$\delta>0$.

As usual for ill-posed operator equations~(\ref{eq:opeq}), we consider smoothness \emph{relative} to the given
operator~$A$. Hence we assume that solution smoothness is given in
terms of source sets as follows:
 \begin{de}
   [source set]\label{de:smoothness}
   There is an index function\footnote{We call a
     function~$\varphi\colon (0,\infty) \to [0,\infty)$ an index
     function, if it is continuous, non-decreasing and satisfies the limit condition~$\lim_{t\searrow
       0}\varphi(t) =0.$ }~$\varphi$ such that the unknown solution
   obeys
   $$
x \in H_{\varphi} := \set{x,\ x= \varphi(A^{\ast}A)v,\ \norm
  {v}{X}\leq 1}.
   $$
 \end{de}

Often the ill-posed operator equation~\eqref{eq:opeq} refers to
Hilbert spaces of functions on a $d$-dimensional domain, see e.g.~the examples
given in the monograph~\cite[Chapt.~5]{MR958691}.
  In many cases the reconstruction rate deteriorates when the spatial
  dimension $d$ grows. This is best seen in the decay rate of the singular values of  embeddings
  between Sobolev spaces in Example~\ref{xmpl:sobolev} below, which exhibits a
  behavior of the form $n^{-a/d}$, where~$n$ is the cardinality of data, and~$a$
  encodes smoothness information. This phenomenon is often called `curse
  of dimensionality'\footnote{The meaning of this notion differs in
    varying places. Here we just mean that the rate as a function of
    the amount of information,  deteriorates exponentially in the
    dimension.}, and it indicates that the amount of computational
  effort that
  is required to find a suitable reconstruction increases
  exponentially with the dimension $d$. However, this may be foiled when
  the leading constant in the asymptotics decays rapidly with
  increasing dimension. In such a case, the curse of dimensionality
  does not necessarily reflect the difficulties in solving the problem at hand.

 In contrast, the associated decay rate may be dimension independent `up to
   some logarithmic factor', which depends on a power of the
   dimension $d$. This phenomenon is also well known, specifically for
   Sobolev embeddings when the underlying spaces are anisotropic. Here
 we highlight the inverse problem of
\emph{reconstruction of copula densities}, representing the correlation structure of a
family of assets, where the \emph{dimensionality} is given by the number~$d$ of
assets (cf.~\cite{Copulas06}). The recent note~\cite{HF23}, for example, outlines the mathematical
model, where the forward operator is a \emph{multivariate integration}
operator. Details for this approach will be given below in Section~\ref{sec:integration}. This model has also been investigated in \cite{Uhlig15} with respect to the stable numerical solution
of the inverse problem of copula density identification, which is  severely hampered by limited computational resources. Regarding the
reconstruction rate, as~$\delta\to 0$, there seems to be no impact
arising from the number of assets. However, as will be made precise below, the
\emph{discretization level} (amount on linear information) may depend
on~$d$, even~\emph{exponentially}.
Consequently, for large~$d$ it may be \emph{infeasible in practice} to reach the region
with the optimal rate of reconstruction.

Therefore, in order to understand the difficulty to solve a
  numerical problem to a given accuracy, a more precise understanding
  is required. This purpose is achieved by studying the tractability of
  such problems. Within the
theory of \emph{Information-based Complexity} there is a long
discussion in this direction, and we refer to the three
volumes~\cite{MR2455266,MR2676032,MR2987170}.
As far as we are aware of, in the literature there is
by now no discussion of the tractability of ill-posed problems, which
takes into account the impact of both occurring  facets: the
dimension~$d$ and the reconstruction rate in terms of~$\delta$. Our goal here is to adapt this
approach to the theory of ill-posed operator equations. Recently, the
authors in~\cite{MR4626413} consider tractability of problems when the
information is noisy. Still this concerns direct problems. In ibid. the
focus is on the comparison of tractability with and without noisy
information. For ill-posed operator equations the presence of noise is
constitutive for the error analysis. Reconstruction rates in the
presense of exact data (often called bias decay) differ significantly from
the rates under noise.

Thus, the outline is as follows. We shall formulate the
  problem, including some examples in Section~\ref{sec:complexity}, and
we give the formal definition of \emph{tractability} in
Section~\ref{sec:tractability}. \pmathe{
The main observation here is presented
  as Theorem~\ref{prop:inverse-direct}. We establish a one-to-one correspondence of the given family of inverse problems, and a
  related family of direct problems, where the class of problem
  elements is defined via the smoothness class from
  Definition~\ref{de:smoothness}. This shows that our notion of 
tractability of a family of  inverse problems is consistent with some companion family of direct ones.}

We discuss examples for operators with
power-type decay of singular values in
Section~\ref{sec:power-type}. Finally, we discuss the family of multivariate
integration problems in Section~\ref{sec:integration}. The tractability
of this problem family is stated as Theorem~\ref{prop:bernd}.

\section{Information complexity, problem formulation}
\label{sec:complexity}

Our focus is on operator equations~(\ref{eq:opeq}) with \emph{compact forward operator}~$A\colon X \to
Y$. These operators allow for a singular value
decomposition (SVD)~$\{s_{j},u_{j},v_{j}\}_{j=1}^{\infty}$ with a  non-increasing infinite
sequence~$\{s_j\}_{j=1}^\infty$ of \emph{singular values},
tending to zero as~$j$ tends to infinity, and sequences~$u_{j}, \
v_{j}$,\ $j=1,2,\dots$, of corresponding \emph{eigenfunctions} for the
self-adjoint operators~$A^{\ast}A$ and~$A A^{\ast}$, respectively.

The standard noise model~(\ref{eq:noise}) is not realistic, as
elements in infinite dimensional Hilbert spaces do not fit numerical computations, and
discretization is required to do so. We follow the standard approach
as presented in~\cite[Chapt.~3]{MR958691}.

\subsection{Discretization}
\label{sec:discretize}

Here we restrict ourselves to \emph{(non-adaptive) linear
  information}, i.e., we choose linear functionals~$L_{1},\dots,
L_{n}\in Y^{\ast}$ and build the~$n$-vector
\begin{equation}
  \label{eq:Nn}
  N_{n}(\yd) := \lr{L_{1}(\yd),\dots,L_{n}(\yd)},\quad \yd\in Y,
\end{equation}
which we shall call information~$N_{n}$. For suitable reconstructions $R$ to the unknown elements $x \in X$ we then may use some arbitrary mapping~$\psi$ into the space $X$, acting
on~$N_{n}(\yd)\in \real^{n}$. Thus the admissible reconstructions are
given as
\begin{equation}
  \label{eq:R}
 \xkd n :=  R(\yd) = \psi(N_{n}(\yd)),\quad \yd \in Y.
\end{equation}

Such approach to ill-posed problems which mimics Information-based
Complexity was first used in~\cite{MR3574546}, and we follow this
pathway.

 We
 consider the worst-case error for any such
 reconstruction~$R$ which uses information of cardinality at most~$n$,
 uniformly on source sets~$H_{\varphi}$, given as in
 Definition~\ref{de:smoothness}.
 At any instance~$x\in
 H_{\varphi}$ the error is given as
 \begin{align*}
   \label{eq:err-x}
   e(R,x,\delta)&:= \sup_{\norm{\xi}{}\leq 1}\norm{ x - R(Ax + \delta\xi)}{X},
   \intertext{and the error uniformly for~$x\in H_{\varphi}$ is given
                  as}
    e(R,H_{\varphi},\delta)&:= \sup_{x\in H_{\varphi}} e(R,x,\delta).
 \end{align*}
 The minimal error~$e_{n}(H_{\varphi},\delta)$ is then the minimum over
 all reconstructions~$R$ using information of cardinality at most~$n$.

\subsection{Problem formulation}
\label{sec:problem}

 A fundamental observation from~\cite{MR3574546} is here rephrased as
 follows.
 \begin{prop}\label{pro:MP}
   Suppose that the operator~$A$ has an SVD with singular values
   $\{s_j\}_{j=1}^\infty$, and that smoothness is given as in
   Definition~\ref{de:smoothness} with an index function~$\varphi$.
   For~$j \in\nat$ we have that
   $$
e_{j}(H_{\varphi},\delta) \geq \varphi\lr{s_{j+1}^{2}}.
$$
\end{prop}
\begin{rem}
  We mention that the assertion of Proposition~\ref{pro:MP} is stated in  \cite[Theorem~1]{MR3574546} for the \emph{Gelfand
    widths}, but by virtue of~\cite[Lemma~1]{MR3574546} these coincide with the singular values.
\end{rem}
\begin{rem}
  In the noiseless case ($\delta=0$) the right hand side above is attained by the error of spectral
  cut-off
  $$
 R(\yd) := \sum_{i=1}^{j} \frac 1 {s_{i}} \scalar{\yd}{v_{i}} u_{i}.
  $$
  uniformly on the smoothness class~$H_{\varphi}$, taking into account
  the $j$ largest singular values. In this case the
  information~$N_{j}$ is based on using the first
  eigenfunctions~$v_{1},\dots,v_{j}$ of the operator~$A A^{\ast}$.
\end{rem}

The above lower bound from Proposition~\ref{pro:MP} is contrasted to the well-known upper bound from
the Melkman-Micchelli construction. We denote by~$e(H_{\varphi},\delta)$ the best
possible accuracy of any reconstruction making use of the full
data~$\yd$. Referring to~\cite{MR2458286} we have
\begin{equation}
  \label{eq:melkman}
  e(H_{\varphi},\delta) \leq \varphi\lr{\Theta^{-1}(\delta)},
\end{equation}
where~$\Theta(t) := \sqrt t \varphi(t)$ is the companion index
function to $\varphi$.

Consequently, the rate~$\delta \mapsto \varphi\lr{\Theta^{-1}(\delta)}$ is the best possible reconstruction rate. It represents the asymptotic regime,
and hence the question is how far discretization needs to go, in order
to reach this regime.

Therefore it is interesting to discuss the index
\begin{equation}
  \label{eq:kast}
  k_{\ast}(\delta) :=
  \begin{cases}
    1, & \text{if}\ \Theta(s_{1}^{2}) \leq   \delta,\\
    \max\set{k \in \nat:\quad \Theta(s_{k}^{2}) >  \delta},& \text{otherwise},
  \end{cases}
\end{equation}
which describes the minimal level of discretization required to achieve order
optimal regularization.

\begin{prop}
  For the cardinality~$k_{\ast}$ from~(\ref{eq:kast}) we have that
  $$
\varphi\lr{s_{k_{\ast}+1}^{2}} \leq
\varphi\lr{\Theta^{-1}(\delta)}.
  $$
\end{prop}
\begin{proof}
  By construction we see that
$
\Theta(s_{k_{\ast}+1}^{2}) \leq \delta.
$
Hence
$$
\varphi\lr{s_{k_{\ast}+1}^{2}} \leq
\varphi\lr{\Theta^{-1}(\delta)},
$$
by the monotonicity of the index function~$\varphi$.
\end{proof}
 Thus, from~$k_{\ast}$ on we see the
asymptotic regime given in~(\ref{eq:melkman}) with the
rate~$\varphi\lr{\Theta^{-1}(\delta)}$, and the size of~$k_{\ast}$ is
a measure of the computational difficulty of the problem at hand.
We highlight this by the following first example.
\begin{xmpl}
  [moderately ill-posed operator] \label{xmpl:moderate} Here we focus on the error
  behavior as~$\delta\to~0$ for a problem with operator whose
  singular values tend to zero polynomially. Precisely, fix~$A$
  with~$s_{j}(A) \sim j^{-a}\;(j \in \nat, a>0)$. Let us assume power type smoothness as
  in Definition~\ref{de:smoothness} for a function~$\varphi(t) =
  t^{p}, \ t>0$, for
  some exponent $p>0$. In this case we see
  that~$k_{\ast}(\delta)\asymp \lr{\frac 1
    \delta}^{\frac{1}{2a(p+1/2)}}$ as $\delta\to 0$, such that the discretization
  level~$k_{\ast}$ increases polynomially in~$1/\delta$, which is
  assumed to be feasible. The asymptotically optimal rate will the be
  seen as~$\delta\to \delta^{p/(p+1/2)}$, regardless of the exponent~$a$.
\end{xmpl}

We point out the following. When the  solution
smoothness is measured in terms of source conditions, as this is done
in Definition~\ref{de:smoothness}, the obtained error bounds are
given in terms of the corresponding index functions, and hence these
are independent of the decay rates of the singular values of the
governing operator~$A$. In contrast, the cardinality~$k_{\ast}$ very
well depends on the singular values,  as well as the noise
level~$\delta$. The question that we are to address is the following: If
we need a discretization level~$k_{\ast}$ in order to enter the
asymptotic regime, will this level be feasible for a problem at hand?
For further motivation we present the next example, which is a slight
  variation of the first one.
\begin{xmpl}
  [mildly ill-posed operator]\label{xmpl:mild}
  Here we consider a similar problem as in
  Example~\ref{xmpl:moderate}, and we assume power type smoothness as
  for a function~$\varphi(t) =  t^{p}, \ t>0$, for
  some exponent $p>0$. However, the singular values of the
  operator~$A$ tend to zero slowly. Precisely, fix~$A$
  with~$s_{j}(A) \sim 1/\log (j)\;(j \in \nat)$.
  We need to describe~$k_{\ast}$ based on formula~\eqref{eq:kast} and
  using the companion function~$\Theta(t) = t^{p+1/2}$ to $\varphi$.
  Again, the order optimal rate is given as~$\delta \to
 \delta^{p/(p+1/2)}$,
 but here we have
  $$
 k_{\ast}(\delta) \sim  \exp\lr{\lr{\frac 1 \delta}^{1/(2p+1)}}\quad \mbox{as} \quad
 \delta\to 0.
 $$
 Thus~$k_{\ast}$ is exponential in~$1/\delta$.
 For a noise level~$\delta:= 10^{-4}$ and with an exponent $p=1/2$
 in the index function $\varphi$, which characterizes the solution smoothness
 as $x\in\range(A^{\ast})$,  we have that $k_{\ast}\sim  e^{100}$. This is certainly not feasible.

 We conclude that it is difficult (intractable) to
 discretize ill-posed problems with mildly ill-posed operator to
 enter the asymptotic regime, even if the noise level $\delta$ is moderate.
 \hfill \fbox{}\end{xmpl}

For $d$-dimensional ($d$-variate) situations, singular values and hence the
index~$k_{\ast}$ depend on both the dimension $d>1$ and the noise
level $\delta>0$.  We present in this context another illustrative example.
\begin{xmpl}[Sobolev embeddings]\label{xmpl:sobolev}
 Here we shall consider an ill-posed problem when the forward operator acts along a given scale of Sobolev spaces~$W^{\msp}_{2}(\Omega)$ on  some bounded
$C^{\infty}$-domain~$\Omega \subset \real^d$ (boundary conditions may be
imposed). This means that $d$ is the spatial dimension of the
domain~$\Omega$.

The spaces $W^{\msp}_{2}(\Omega)$ form, for a given interval $[-a,a]\;(a>0)$, a scale of Hilbert spaces
(see, e.g., \cite{Neubauer88}) with two-sided estimates
generated by some unbounded self-adjoint operator, say~$L$.
For~$0 \leq \nu \leq a$
we have that~$x\in W^{\nu}_{2}(\Omega) $ exactly if~$x
\in\domain(L^{\nu})$ belongs to the domain of~$L^{\nu}$. Negative
smoothness is given by duality.

The above phrase `along a scale' means that~$A$ acts
from~$W^{-a}(\Omega)$ to~$Y$, and 
there are constants~$0 < m \leq M < \infty$, for which
\begin{equation}
  \label{eq:link}
 m \norm{x}{-a}  \leq \norm{Ax}{Y} \leq M \norm{x}{-a},\quad x\in W^{-a}_{2}(\Omega).
\end{equation}
Also, we assume that the given solution smoothness is~$x\in
W^{p}_{2}(\Omega) = \range(L^{-p})$, and hence there is some source element~$v$ with~$x= L^{-p}v$. For details we also refer to the monograph~\cite[Chapt.~4.9]{MR1328645}.

The crucial tool for understanding this situation is given by Douglas'
range inclusion theorem, see its formulation in~\cite{MR3985479}. Thus we see that condition~(\ref{eq:link}) is equivalent
to~$\range(\lr{A^{\ast}A}^{1/2}) = \range(L^{-a})$ (with corresponding
norm bounds). Using Heinz'
inequality (\cite[Prop.~8.21]{MR1408680}) this yields for~$0 < p \leq a$
that~$\range(\lr{A^{\ast}A}^{p/(2a)}) = \range(L^{-p})$. The link
condition also has a consequence given by Weyl's monotonicity theorem,
see~\cite[Chapt.~III.2.3]{MR1477662}, and this reads as $m s_{j}(L^{-a}) \leq
s_{j}(A)\leq M s_{j}(L^{-a})$ for $j \in \nat$, where~$s_{j}:= s_{j}(A)$
and~$s_{j}(L^{-a})$ denote the singular values of~$A$ and~$L^{-a}$,
respectively.  It will be important
to know that asymptotically~$s_{j}(L^{-a}) \asymp j^{-a/d}$
as~$j\to\infty$. Hence, the dimensionality of the domain enters the
decay rate of the singular values.

For convenience let us consider the index function~$\varphi(t):= t^{p/(2a)}$
for describing the solution smoothness and its companion function~$ \Theta(t) := \sqrt t \varphi(t) =
t^{(a+p)/(2a)}$. Moreover, for simplicity, we assume that the singular
value decomposition of~$A$ is available. Then we can
use this to build reconstructions
\begin{equation}
  \label{eq:xkd}
\xkd n :=
\sum_{j=1}^{n}\frac{1}{s_{j}}\scalar{\yd}{v_{j}}u_{j}\quad (n \in \nat)
\end{equation}
from the singular value decomposition. Straightforward calculations
yield the order optimal error bound
\begin{equation}
  \label{eq:error-bound}
  \norm{x - \xkd n}{X} \leq \varphi\lr{s_{n+1}^{2}} +
  \frac{\delta}{s_{n}},\quad (n \in \nat)\,.
\end{equation}
For the cardinality~$k_{\ast}$ from~(\ref{eq:kast}) this  gives the error bound
$$
  \norm{x - \xkd {\kast}}{X} \leq 2  \varphi\lr{s_{\kast}^{2}}.
  $$
 What will be the size of~$k_{\ast}$ in terms of the
  dimension~$d$ and of the noise level~$\delta$? Taking into account Weyl's
  monotonicity theorem and the present structure of the function~$\Theta$,
  we see that, for small~$0 <\delta < 1$,
  \begin{equation}
    \label{eq:kast-dimension}
    \kast(\delta,d) \asymp \lr{\frac 1 \delta}^{\frac d {a+p}} \quad \mbox{as} \quad  \delta\to 0,
  \end{equation}
with an implied rate~$\delta \to \delta^{p/(a+p)}$ as~$\delta\to
  0$. This rate is order optimal under the present conditions, but we
  stress that the discretization level~$\kast$ depends on~$1/\delta$
  with a power having the dimension~$d$ in its enumerator.
Assuming that the leading constant, say~$C(d)>0$ in~(\ref{eq:kast-dimension}) is
bounded away from zero, say~$C(d)\geq 1$ for simplicity, and for~$a=p=1/2$ we find that~$k_{\ast}\asymp
100^{d}$ in the case of a moderate noise level $\delta=10^{-2}$. Such values $k_{\ast}$, however, are not feasible for large dimensions.
Hence the problem is difficult (intractable), even for moderate
  values of~$\delta>0$,  when the dimension $d$ is large enough.
\end{xmpl}

\section{Tractability}
\label{sec:tractability}

Following the
monographs~\cite{MR2455266}--\cite{MR2987170}, we will now suggest a definition that tries to formalize the notion of `difficulty' that we have highlighted in the
Examples~\ref{xmpl:mild} and~\ref{xmpl:sobolev}.

\subsection{Tractability of families of ill-posed equations}
\label{sec:tratcatbility-inverse}

We want to capture both, difficulties for small level~$\delta$, as
well as difficulties for large spatial dimensions~$d$.
In order to capture the dimensionality~$d$ of the problems, the concept
of tractability is based on the following construction: For~$d\in\nat$
we introduce a family of linear (multivariate) problems by considering an associated family of operator equations~\eqref{eq:opeq} with compact linear operators $A:=A_{d}$.
Having such a family $A_{d}\;(d \in\nat)$, with corresponding singular
values~$s_{j}(A_{d})$ fixed, the
following definition seems to be appropriate, for the index~$\kast =
\kast(\delta,d)$ from~(\ref{eq:kast}).

\begin{de}   [weak tractability] \label{def:trintr}
  We call the family of operator equations \eqref{eq:opeq} with compact linear operators $A:=A_{d}\: (d \in\nat)$ {\sl weakly tractable} if we have
  \begin{equation} \label{eq:Qlimit}
\lim_{d + 1/\delta\to \infty} Q(\delta,d) = 0 \quad \mbox{for} \quad Q(\delta,d):= \frac{\log(k_{\ast}(\delta,d))}{d +
  1/\delta}\,.
  \end{equation}
Otherwise we call it \emph{intractable}.
\end{de}
\begin{rem}
  The above limit~$d + 1/\delta\to\infty$ means that for each
  subsequence~$\{(\delta_{k},d_{k})\}_{k=1}^\infty$ with~$d_{k} +
  1/\delta_{k}\to\infty$ the  quotient $Q(\delta_k,d_{k})$ tends to zero as $k \to \infty$.
If there are subsequences such that the quotient  $Q(\delta_k,d_{k})$ is bounded away from zero, then intractability is seen.

  In
  particular, intractability (non-vanishing limit~$Q$) may occur if either
  \begin{enumerate}
  \item[(I)] $d \in \nat$ is fixed and~$\delta \to 0$ (intractability in~$\delta$), or
  \item[(II)] $\delta>0$ is fixed and~$d\to\infty$ (intractability in~$d$).
       \end{enumerate}
    By definition, a problem is intractable if~$k_{\ast}$ is exponential in~$d$
    \emph{or}~$1/\delta$. Of course, often a sequence
of problems may be intractable in both factors, simultaneously.
\end{rem}
\begin{rem}
  Formally, we just have a family of problems with the corresponding
  sequence of compact operators $A_{d}\;(d\in\nat)$, and
  these do not need to have anything in common. In order to interpret
  the tractability vs.~intractability, ``the reader must be convinced that
  the definition of~$A_{d}\;(d \in \nat)$  properly models the same
  problem for varying dimensions''
  (see~\cite[p.~101]{MR1263379}). This is the case for the problem
  studied here.
\end{rem}
With Definition~\ref{def:trintr} at hand, we may rephrase that the problem family of
Example~\ref{xmpl:mild} is intractable in~$\delta$, whereas the
problem family in Example~\ref{xmpl:sobolev} is intractable in~$d$.
\begin{rem}
We add that there is a zoo of modifications of the notion of
tractability, often emphasizing power-type behavior, or other special
features, see ~\cite[Chapt.~8]{MR2455266}. Here we constrain to the weakest notion of tractability.
\end{rem}

\subsection{Relation to tractability of direct problems}
\label{sec:tractability-direct}

We now relate the tractability for ill-posed equations to the one as
studied in Information-based Complexity.

\pmathe{We recall that initially we are given an operator
  equation~$A\colon X \to Y$ as in~(\ref{eq:opeq}). The goal
  was to solve the \emph{inverse problem} under the knowledge that the
  unknown solution belongs to the set~$H_{\varphi}$ from
  Definition~\ref{de:smoothness}. Then we turned to
  families~$A_{d}\colon X(d) \to Y(d)\ (d\in\nat)$ of such equations
  in order to treat the tractability for families of inverse problems.}

\pmathe{In Information-based Complexity one typically also starts with
such operator equation as in~(\ref{eq:opeq}). However, the goal is to
approximately construct~$A x,\ x\in F\subset X$, where~$F$ is the set
of problem elements. The construction is based on
information~$N_{n}(x)$, similarly as in~(\ref{eq:Nn}), and with error
measured in~$Y$.  In this context
the mapping~$A\colon F \to Y$ is called \emph{solution operator},
see~\cite[Chapt.~4]{MR2455266}. 
Typically, the set~$F:= B_{X}$ is the unit ball in~$X$. We agree to call this the \emph{direct
  problem}.
Again,
tractability is then defined for families of such direct problems.
}

\pmathe{
  Here we shall construct a ~\emph{companion problem} ( `direct' in the sense of
  Informa\-tion-based Complexity) to the inverse problem, taking into
  account that the set~$B_{X}$ of problem elements coincides
  with~$H_{\varphi}$ from Definition~\ref{de:smoothness}. 
}

The set~$F:= H_{\varphi} \subset X$ of
problem elements, is
seen to be the unit ball in the Hilbert space~$X_{\varphi}$, given as
\begin{equation}
  \label{eq:H-phi}
  X_{\varphi} := \set{x \in \ker^{\perp}(A),\quad x=\varphi(A^{\ast}A)v,\ \norm{v}{X} <\infty},
\end{equation}
endowed with the natural scalar product
\begin{equation}
  \label{eq:scalar}
  \scalar{x}{y}_{X_{\varphi}}:= \scalar{v}{w}_{X},
\end{equation}
where the elements~$v,w$ are the unique source elements,
because~$X_{\varphi}\subset \ker^{\perp}(A)$ is restricted to the orthogonal
complement to the kernel of~$A$. Actually, the spaces~$X_{\varphi}$
  with~$\varphi$ being index functions generate scales of Hilbert
  spaces, and we refer to the study~\cite{MR1984890}.

\pmathe{We attach to the inverse problem the companion problem,
  written as~$A^{\varphi}\colon
X_{\varphi}\to Y$, to distinguish from the direct problem~$A \colon X
\to Y$}. Given~$x\in H_{\varphi}$, the goal is to
approximately compute~$A^{\varphi} x\in Y$. The error for any algorithm~$R_{k} =
\psi(N_{k}(x))\colon X\to Y$, using at most information of
cardinality~$k$,
is given as
\begin{equation}
  \label{eq:error-ibc}
  e(A^{\varphi},R_{k}):= \sup_{x\in H_{\varphi}}\norm{A^{\varphi} x - R_{k}(x)}{Y}.
\end{equation}
\pmathe{Suppose that we have a family~$A_{d}^{\varphi}\colon X_{\varphi}(d)
\to Y(d)\; (d\in\nat)$ at hand, and we stress that the unit balls also
depend on~$d$, such that we have~$H_{\varphi}(d)$.} For  this absolute error criterion the (weak) tractability is derived
from the information complexity~$n(\varepsilon,d)$, given as
\begin{equation}
  \label{eq:ned}
  n(\varepsilon,d) := \max\set{k,\quad \inf_{R_{k}}\;e(A_{d}^{\varphi},R_{k}) > \varepsilon }.
\end{equation}
(We tentatively assume that~$\varepsilon>0$ is small enough, such
that~$\varepsilon < \norm{A_{d}}{X_{\varphi}\to Y}$.)
According to~\cite[~\S~4.4.2]{MR2455266}  the problem is (weakly)
tractable if
\begin{equation}
  \lim_{d + 1/\varepsilon\to\infty} \frac{\log(n(\varepsilon,d))}{d +
    1/\varepsilon} =0,
\end{equation}
otherwise it is intractable. For the companion
problems~$A_{d}^{\varphi}$ the information
complexity~$n(\varepsilon,d)$ is characterized by the singular
values~$s_{d,k},\ k\in\nat$
of the maps~$A_{d}^{\varphi}\colon X_{\varphi}(d)\to Y(d)$ via
\begin{equation}
  \label{eq:ned-ld}
  n(\varepsilon,d) =
\max\set{k,\quad s_{d,k}>\varepsilon}.
\end{equation}
\pmathe{The main observation is that the family of inverse problems and the
family of companion problems share the same tractability behavior.}
\begin{thm}\label{prop:inverse-direct}
  Suppose that we have a problem family~$A_{d}\; (d\in\nat)$
  with~$\rk(A_{d})=\infty$ and
  SVD $(\tilde s_{d,k}, u_{d,k},v_{d,k}),\ k\in\nat$, with smoothness given as in
  Definition~\ref{de:smoothness}.
Then we have for~$\kast(\delta,d)$ from~(\ref{eq:kast}) the identity
$$
 \kast(\delta,d) = n(\delta,d).
 $$
 Thus, the ill-posed problem family~$A_{d}\; (d\in\nat)$ is tractable
 if and only if the family of companion problems~$A_{d}^{\varphi}\ (d\in\nat)$ is tractable.
\end{thm}
\begin{proof}
  We need to find the SVDs of the
  maps~$A_{d}^{\varphi}\colon X_{\varphi}(d)\to Y(d)$. Knowing the SVD of~$A_{d}\colon
  X(d)\to Y(d)$ we argue as follows. For every~$d\in\nat$, and corresponding
  SVD~$(\tilde s_{d,j},u_{d,j},v_{d,j}),\
  j=1,2,\dots$ of the map~$A_{d}$, we see that
  \begin{equation}
    \label{eq:svd-Ad}
    A_{d} \varphi(A_{d}^{\ast}A_{d}) u_{d,j} =
    \tilde{s}_{d,j}\varphi(\tilde s_{d,j}^{2}) v_{d,j},\quad j\in\nat.
  \end{equation}
  Consequently, the map~$A_{d}^{\varphi}\colon X_{\varphi}(d)\to Y(d)$ has the
  following SVD, namely, in terms of  the companion
  function~$\Theta$ to~$\varphi$,
  $$
  \lr{\Theta(\tilde{s}^{2}_{d,j}), \varphi(A^{\ast}A) u_{d,j},v_{d,j}},\ j\in\nat.
  $$
  Notice, that by construction, for every~$d\in\nat$ the
  elements~$\varphi(A_{d}^{\ast}A_{d}) u_{d,j}$
  form an orthonormal basis in~$X_{\varphi}$. Thus,  we see
  that
  \begin{equation}
    \label{eq:n-kast}
     n(\delta,d)
  = \max\set{k,\quad \Theta(\tilde{s}_{d,k}^{2})>\delta}
  = \kast(\delta,d).
  \end{equation}
  The proof is complete.
\end{proof}
\pmathe{Above we have established a one-to-one correspondence between the
tractability of \emph{inverse problems} with smoothness given as
in Definition~\ref{de:smoothness}, and families of \emph{companion problems}~$A_{d}^{\varphi}\colon
X_{\varphi}(d)\to Y(d)$.
As a counterpart we have the corresponding
family~$A_{d}\colon X(d) \to Y(d)$ of \emph{direct problems (in the sense of
Information-based Complexity)}.
These families of problems are related as follows.}
\begin{cor}\label{cor:inverse-direct}
  Suppose that we have a family~$A_{d}\; (d\in\nat)$ of operator equations. The following holds true.
  \begin{enumerate}
  \item If the family of direct problems~$A_{d}\colon X(d)\to Y(d)$ is \emph{tractable}, then this
    holds true for the corresponding family~$A_{d}^{\varphi}\colon
    X_{\varphi}(d)\to Y(d)$. Hence, the family of inverse problems is
    tractable regardless of the smoothness given in
    Definition~\ref{de:smoothness}.
    \item If there is smoothness as in Definition~\ref{de:smoothness} with index function~$\varphi$, for which
      the family of inverse problems is intractable, then so is the
      family of companion problems~$A_{d}^{\varphi}\colon
      X_{\varphi}(d) \to Y(d)$. Consequently, the
      family of direct problems~$A_{d}\colon X(d)\to Y(d)$ is \emph{intractable}.
  \end{enumerate}
\end{cor}
\begin{proof}
  [Sketch of the proof]
  Let us temporarily abbreviate~$k^{\varphi}_{\ast}(\delta,d)=
  k_{\ast}(\delta,d)$ to highlight the dependency on the smoothness,
  given in terms of the index function~$\varphi$. Furthermore, we
  confine the the case  when~$k^{\varphi}_{\ast}(\delta,d)\to\infty$
  as~$d + 1/\delta\to\infty$. Hence, if~$d + 1/\delta$ is large
  enough, we may assume
  that~$\varphi(\tilde s_{d,k^{\varphi}_{\ast}}^{2})\leq 1$. In this case we
  learn from~(\ref{eq:n-kast}) that~
    $$
    k^{\varphi}_{\ast}(\delta,d) =  \max\set{k,\ 
      \Theta(\tilde{s}_{d,k}^{2})>\delta}
    \leq  \max\set{k,\ \tilde{s}_{d,k}^{2}>\delta} = n(\delta,d).
    $$
  The assertions of the corollary are a direct consequence of this relation.
\end{proof}
\begin{rem}
  A less technical argument is as follows. \pmathe{By construction, we have
  the continuous embedding~$X_{\varphi}(d) \hookrightarrow X(d)$.} After rescaling we may
  assume that~$H_{\varphi}(d) \subseteq B_{X(d)}$, where~$B_{X(d)}$ is the unit ball in
  X(d). Thus, if the family is tractable on~$B_{X(d)}$ then it will also
  be tractable on~$H_{\varphi}(d)$, which corresponds to the
  \pmathe{tractability of the} family of
  inverse problems.  Similarly, if for some~$\varphi$ the
 \pmathe{ family of inverse problems} will be intractable, then the family of direct problems
  \pmathe{(on~$B_{X(d)}\supset H_{\varphi}$)} will also be intractable.
\end{rem}

\medskip
We close this section with the following observation. Often it may be
hard to get grip on the exact value for~$\kast$. Clearly,
tractability/intractability is an asymptotic
property. If~$\kast(\delta,d)$ is uniformly bounded then tractability
is clearly seen. So, the interesting case is
when~$\kast(\delta,d)\to\infty$ as~$d + 1/\delta\to\infty.$
Then the
following auxiliary result can be used.

\begin{prop}\label{prop:asymptotics}
  Let~$a >0$ and~$b\in\real$. The following holds true,
  whenever~$a\kast +b$ is positive.
  \begin{enumerate}
  \item[(i)] The problem is tractable if and only if
    $$
    \lim_{d + 1/\delta\to\infty}\frac{\log \lr{a k_{\ast}(\delta,d) + b}}{d + 1/\delta} = 0.
    $$
  \item[(ii)] If along a subsequence $\{(\delta_{k},d_{k})\}_{k=1}^\infty$ we
    have that
    $$Q(\delta_{k},d_{k})\geq \underbar c >0\,,$$
    then   $\;\frac{\log \lr{a k_{\ast}(\delta,d) + b}}{d + 1/\delta}\geq
    \underbar c/2>0\;$ for $\;d_{k} + 1/\delta_{k}\;$ large
    enough. Hence the problem family is intractable.
  \end{enumerate}
\end{prop}
\begin{proof}
  [Sketch of a proof] The analysis is simple, and we just hint that
  $$
  \frac{\log \lr{a k_{\ast} + b}}{d + 1/\delta}
  =  \frac{\log \lr{a k_{\ast}\lr{1  + \frac b {a\kast}}}}{d + 1/\delta}
  = \frac{\log(\kast)}{d + 1/\delta} + \frac{\log(a)}{d + 1/\delta}
  + \frac{\log\lr{1  + \frac b {a\kast}}}{d + 1/\delta}
  $$
  The last two terms on the right tend to zero as~$d +
  1/\delta\to\infty$, and hence have no impact on the asymptotics.
\end{proof}
\section{Power-type decay of the singular values}
\label{sec:power-type}

Here we discuss in more detail another example, showing that the joint
limit~$d + 1/\delta\to \infty$ may be crucial. We assume a power-type decay
of the singular values of the operator family~$A_{d}\;(d\in\nat)$. Specifically we assume that for some
power~$a>0$, and a leading term\footnote{We chose the parametrization
  in terms of~$1/d$, because then the constant~$\cd$ may be decreasing
with increasing dimension~$d$.}~$\cd$ there are constants~$0 < m \leq
M <\infty$ such that we have
\begin{equation}
  \label{eq:cd-power}
   m \;\cd j^{-a/d} \leq  s_{j}(A_{d}) \leq M \;\cd j^{-a/d}\qquad (j\in\nat).
 \end{equation}

 For the tractability analysis the following observation is useful:
Looking at the definition of~$\kast$ in~(\ref{eq:kast}), and taking into account the bounds from~(\ref{eq:cd-power}), we see the following:
 \begin{enumerate}
 \item \label{it:kast-low}
   If $j < \lr{\frac{m^{2}\cd^{ 2}}{\Theta^{-1}(\delta)}}^{d/2a}$
   then~$\kast(\delta,d)\geq j$.
   \item \label{it:kast-up} If $j \geq  \lr{\frac{M^{2}\cd^{ 2}}{\Theta^{-1}(\delta)}}^{d/2a}$
   then~$\kast(\delta,d)\leq j$.
 \end{enumerate}
 By taking largest and smallest such values~$j$, respectively, we end
 up with the following two-sided bound.
 \begin{equation}
   \label{eq:kast-two}
 \lr{\frac{m^{2}\cd^{ 2}}{\Theta^{-1}(\delta)}}^{d/2a} - 1  \leq
 \kast(\delta,d) \leq  \lr{\frac{M^{2}\cd^{
       2}}{\Theta^{-1}(\delta)}}^{d/2a} +1.
\end{equation}
Thus, in order to see \emph{tractability} for a given situation, we
need to consider the right hand side, whereas the left hand side will
be used to show \emph{intractability}. In this context, the assertion of
Proposition~\ref{prop:asymptotics} may be used.

We discuss the impact of the behavior of the leading constant~$\cd$
on tractability/intractability. We will distinguish three benchmark situations.
If~$\cd$ is bounded away from zero
as~$d\to\infty$, intractability occurs. If~$\cd$ goes to zero quickly,
at least linear in~$1/d$, then tractability is seen. In the
intermediate cases, specifically if $\cd$ is sublinear in~$1/d$, then
intractability can be seen for low smoothness.
We give details for these cases, next.
\begin{description}
\item[a) $\mathbf{\cd}$ is bounded from below]
In this case the analysis is particularly simple.
\begin{prop}
If~$\cd \geq \underbar c >0$ then the problem family $A_{d}\;(d\in\nat)$ is
intractable.
\end{prop}
\begin{proof}
  If $\cd$ is bounded away from zero as~$d\to\infty$, then the lower
  bound in~(\ref{eq:kast-two}) yields for~$\delta_{0} <
  \Theta(m^{2}\underbar c^{2})$ an exponential increase in~$d$, and
  the problem is \emph{intractable in~$d$}.
\end{proof}
\item[b) $\mathbf{c(1/d)}$ is at least linear
  in~$\mathbf{1/d}$]
  In this case the following is seen.
  \begin{prop}\label{prop:linear}
    If the function~$t \mapsto c(t)$ is  at least linear, i.e.,\ $c(t) \leq \bar c t$ for some constant~$\bar c$, then the
    problem family $A_{d}\;(d\in\nat)$ is tractable.
  \end{prop}
  For the proof we need the following simple but important
  observation, which takes into account the special form of the
  function~$\Theta$.
  \begin{lemma}\label{lem:theta}
    For each constant~$C_{0}>0$ there is~$\delta_{0}>0$ such
    that
    $$
    C_{0}^{2}\,\delta^{2} \leq \Theta^{-1}(\delta)\quad \text{whenever}\ 0 < \delta < \delta_{0}.
    $$
  \end{lemma}
  \begin{proof}
    Clearly, for the given smoothness function~$\varphi$ we
    find~$\delta_{0}>0$ such that
    $$
    C_{0}\,\delta\, \varphi(C_{0}^{2}\,\delta^{2}) \leq \delta,\quad
    \text{for}\ 0 < \delta \leq \delta_{0}.
    $$
    By the definition of~$\Theta$ this
    yields~$\Theta\lr{C_{0}^{2}\,\delta^{2}} \leq \delta$ for~$0 <
    \delta \leq \delta_{0}$, and the proof can easily be completed.
  \end{proof}
  \begin{proof}
    [Proof of Proposition~\ref{prop:linear}]
We distinguish two cases, relevant for the analysis. First let us assume that~$\delta d\geq
    \underbar c >0$. Then we apply Lemma~\ref{lem:theta} with~$C_{0}:=
    \frac{M \bar c}{\underbar c}$.  This gives
    $$
    M^{2}\frac{\bar c^{2}}{d^{2}} \leq C_{0}^{2}\,\delta^{2}
    \leq \Theta^{-1}(\delta)\quad \text{for}\ \underbar
      c/d\leq  \delta \leq \delta_{0},
    $$
    which by virtue of~(\ref{eq:kast-two}) implies~$\kast\leq 2$ in this case.

    Next, let us assume that~$\delta d \to 0$, in particular~$\delta\to
    0$. Specifically we may assume ~$\delta d
    \leq \bar c$, and~$\delta \leq \delta_{0}$.
By using the right hand side in~(\ref{eq:kast-two}), we can bound from above as
    $$
    \kast -1  \leq \lr{\frac{M^{2}\cd^{2}}{\Theta^{-1}(\delta)}}^{d/2a} \leq
    \lr{\frac{M^{2}\bar c ^{2}}{d^{2}\Theta^{-1}(\delta)}}^{d/2a} .
    $$
    Lemma~\ref{lem:theta} with $C_{0}:= M \bar c$
    yields $\delta_{0}>0$ such that
    $$
    M^{2}\,\bar c^{2}/\Theta^{-1}(\delta)\leq 1/\delta^{2} \quad(\delta \leq \delta_{0}),
    $$
    and we can estimate as
    \begin{align*}
      \frac{\log (\kast(\delta,d)-1)}{d + 1/\delta}
      & \leq \frac{\delta d}{2a} \log\lr{\frac{M^{2}\bar c
        ^{2}}{d^{2}\,\Theta^{-1}(\delta)}} \leq  \frac{\delta d}{2a}
        \log\lr{\frac 1{(\delta d)^{2}}}\,,
    \end{align*}
    which implies that $\frac{\log (\kast(\delta,d)-1)}{d + 1/\delta}$ tends to zero as $\delta d \to 0$.
    By virtue of Proposition~\ref{prop:asymptotics} the proof is
    complete, and  we have tractability.
  \end{proof}

\item[c) $\mathbf{\cd}$ is sublinear]
\label{sec:cd-sublinear}
  This case exhibits an interesting feature. To be precise we assume
  here that the function~$1/d\to \cd$ is the restriction of a strictly
  increasing continuous sublinear\footnote{An
      index function~$f$ is called sublinear if the function~$t
      \mapsto t/f(t)$ is an index function.} index function~$c\colon (0,\infty)\to (0,\infty)$.
  \begin{prop}
    Suppose that the function~$t\to c(t)$ is strictly
      increasing and sublinear. For a
    constant~$C > 1/m^{2}$, with~$m$ from~(\ref{eq:cd-power}), let~$\varphi$ be any index function which satisfies
    \begin{equation}
      \label{eq:indexlow}
      \varphi(t) \geq \frac{c^{-1}(\sqrt{Ct})}{\sqrt t}\quad (0 < t
      \leq 1/C).
    \end{equation}

    Then the problem family is intractable for this smoothness~$\varphi$.
  \end{prop}
  \begin{proof}
    First notice, that the right hand side constitutes an index
    function, since with~$s:= c^{-1}(\sqrt{Ct})$ we see that
    $$ \frac{c^{-1}(\sqrt{Ct})}{\sqrt t} = \sqrt C\, \frac{s}{ c(s)},
    $$
    and the sublinearity of~$c$ applies. 

    If the function~$\varphi$ obeys~(\ref{eq:indexlow}), then with
    letting~$t:= \frac{\cd^{2}}{C}$ we find that
    $$
    \Theta\lr{\frac{\cd^{2}}{C}} \geq \frac 1 d,
    $$
    and hence that
    $$
    \cd^{2} \geq C \, \Theta^{-1}\!\left(\frac{1}{d}\right),
    $$
    which implies
    \begin{equation}
      \label{eq:ct-quotient}
      \frac{m^{2}\cd^{2}}{\Theta^{-1}\!\left(\frac{1}{d}\right)} \geq C {m^{2}} >1\quad (d\in\nat).
    \end{equation}
    Consequently, for the sequence~$(\delta(d),d)$
    with~$\delta(d)=1/d$ we can see
    that~$\kast(1/d,d) +1$ is growing exponentially with $d$,
    because the associated constant~$C m^{2}$ is greater than one. By using Proposition~\ref{prop:asymptotics}, this shows intractability under the
    given smoothness~$\varphi$ from~(\ref{eq:indexlow}).
  \end{proof}
  \begin{xmpl}
    Let us consider the case when~$c(t) =t^{q}$ for some~$0< q < 1$,
    to maintain sublinearity. Then the benchmark smoothness on the
    right in~(\ref{eq:indexlow}) is seen to be~$t \to
    t^{\frac{1-q}{2q}}$. This highlights that we need to have low
    smoothness if~$q$ is close to one, whereas smoothness can be
    arbitrarily large when~$q$ tends to zero.
  \end{xmpl}

\end{description}

\section{Multivariate integration operator}
\label{sec:integration}

The current study was inspired by the investigations in~\cite{HF23} for the operator of $d$-variate
mixed integration $A_{d}\colon L^{2}(0,1)^{d} \to L^{2}(0,1)^{d}\;(d~\in~\nat)$, which is defined
for~$0 < s_{1},\dots,s_{d} < 1$ as
\begin{equation}
  \label{eq:berndmult}
  (A_{d}x)(s_{1},\dots,s_{d}) := \int_{0}^{s_{d}}\dots\int_{0}^{s_{1}}
x(t_{1},\dots,t_{d})\; d t_{1}\cdots dt_{d}.
\end{equation}
It was shown in~\cite{HF23}, and it was evaluated in a
different context also in~\cite{KL02},  that the singular values behave as
\begin{equation}
  \label{eq:sn-multi}
s_{j}(A_d) \asymp C(d) \,\frac{\log^{d-1}(j)}{j} \quad \text{as}
\quad j\to\infty.
\end{equation}

    \begin{rem}
    Formulas of the form~\eqref{eq:sn-multi} possessing the associated
  asymptotics as $j \to \infty$ occur in different applications. For
  example, we have such formulas in~\cite{Schmeisser07} for
  characterizing the $j$-th entropy numbers of an embedding operator
  from Sobolev spaces of dominating mixed smoothness to
  $L^2(0,1)^{d}$, and in~\cite{Tem18} for the Kolmogorov $j$-th-width
  in case of periodic functions over the $d$-dimensional torus
  $\mathbb{T}^d$. Within the present context, the study~\cite{KSU15} is most important.
  These authors analyze approximation numbers (coinciding with singular
  values) of Sobolev embeddings over the $d$-dimensional torus
  $\mathbb{T}^d$.  So, different behavior of~$C(d)$ under the
  auspices of the formula~\eqref{eq:sn-multi} can be seen in the literature.
\end{rem}
The major point in~\cite{HF23} was to stress that the degree of ill-posedness
of the operators $A_d$ from~\eqref{eq:berndmult} with the asymptotics~\eqref{eq:sn-multi} for the singular
values proves to be {\sl one}, regardless of the spatial dimension
$d$. Here, the degree of ill-posedness measures the power-type decay
of the singular numbers of the operators~$A_d$, and we have that
$$
\liminf_{j\to\infty}
\frac{-\log(s_{j}(A_{d}))}{\log(j)}=
\limsup_{j\to\infty}
\frac{-\log(s_{j}(A_{d}))}{\log(j)}
=\lim_{j\to\infty}
\frac{-\log(s_{j}(A_{d}))}{\log(j)}
=1,
$$
see~\cite[Def.~1.1]{HF23}.

  However, such a family $A_{d}\;(d\in\nat)$ of problems may be intractable. Then the
  degree of ill-posedness does not necessarily reflect the difficulty
  for solving the related ill-posed problem, and the decision between tractability and intractability is influenced by the
  behavior of the leading constant~$C(d)$.

   If $C(d) \ge \underbar c >0$, then the problem is obviously
   intractable in $d$, because the function
   $f(t)=\frac{\log^{d-1}(t)}{t}\;(t \ge 1)$ is growing for $t \in
   [1,e^{d-1}]$ and thus $\log(k_*(\delta,d)) \ge d-1$ for
   sufficiently small $\delta>0$.

   The situation may change when~$C(d)$ tends to zero as~$d \to \infty$.
    To this end, let us assume smoothness given by any index function~$\varphi$ with related
  companion~$\Theta$, as described in
  Definition~\ref{de:smoothness}. The following result is relevant.

  \begin{prop} \label{prop:speed}
    Suppose that, for a family~$A_{d}\;(d\in\nat)$ of compact operators obeying a singular value behavior of the form \eqref{eq:sn-multi}, there is a
    constant~$c_{0}>0$ such that
    $C(d) \ge c_{0}\,\lr{\frac e{d-1}}^{d-1}\;(d=2,3,...)$,
    and hence there is some constant~$\underbar c>0$, for which the singular
    values~$s_{j}(A_{d})$ are bounded from below by
   \begin{equation}
      \label{eq:sn-below}
      s_{j}(A_{d}) \geq \underbar c \lr{\frac e
        {d-1}}^{d-1}\frac{\log^{d-1}(j)}{j}\quad (j,d=2,3,\dots)\,.
    \end{equation}
    Then the family~$A_{d}\;(d\in\nat)$ is intractable in~$d$.
  \end{prop}
  \begin{proof}
    By the definition of~$k_{\ast}$ in (\ref{eq:kast}) we can argue as
    follows: If for an index~$l$ we see that~$s_{l}^{2}(A_{d})>
    \Theta^{-1}(\delta)$ then~$k_{\ast}(\delta,d)\geq l$.
    We let~$l:= \lceil e^{d-1}\rceil$, the smallest integer
      larger than~$e^{d-1}$, and we fix some~$\delta_{0}\leq
    1/2$ such that~$4 \Theta^{-1}(\delta_{0}) < \underbar c$.
    Then, for~$d\geq 2$, we can bound
    \begin{align*}
      s_{l}(A_{d}) & \geq \underbar c \lr{\frac e
                     {d-1}}^{d-1}\frac{\log^{d-1}(l)}{l} \\
      &\geq \underbar c \lr{\frac e
                     {d-1}}^{d-1} \frac{\log^{d-1}(e^{d-1})}{e^{d-1}
        +1}\\
      &\geq \frac{\underbar c}{2}  \lr{\frac e {d-1}}^{d-1}
        \lr{\frac  {d-1} e}^{d-1} = \frac{\underbar c}{2}  > \sqrt{\Theta^{-1}(\delta_{0})}.
    \end{align*}
    Therefore~$s_{l}^{2}(A_{d}) >  \Theta^{-1}(\delta_{0})$, and
    hence~$k_{\ast}(d,\delta_{0}) \geq e^{d-1}$. But then, for~$d\geq
    1/\delta_{0}\geq 2$ we see that
    $$
 \frac{\log (k_{\ast}(d,\delta_{0}))}{d + 1/\delta_{0}} \geq
 \frac{d-1}{2d} \geq 1/4,
 $$
 such that the family~$A_{d}\;(d\in\nat)$ of operators is intractable in~$d$.
\end{proof}

Now we return to the family of multivariate integration operators from~\eqref{eq:berndmult}, and
we will show that this class of ill-posed problems is weakly tractable.

  \begin{thm} \label{prop:bernd}
  For the family of compact multivariate integration operators $A_d:L^2(0,1)^d \to L^2(0,1)^d$ defined in \eqref{eq:berndmult} we have for the constant $C(d)$ in \eqref{eq:sn-multi} that
  \begin{equation} \label{eq:CC}
  C(d) \sim \frac{1}{(d-1)!\,\pi^d}\  \left(\asymp \frac{1}{\sqrt{d-1}}\,\lr{\frac e {\pi (d-1)}}^{d-1}\right)\quad \text{as}\quad d\to\infty,
  \end{equation}
 and hence this class of problems is weakly tractable in $d$.
  \end{thm}
  \begin{proof}
  Formula \eqref{eq:CC} is known from~\cite[Thm.~2]{MR3724810}.
  Indeed, this author bounds the singular values of the multivariate problem
  from known bounds for the underlying univariate ones by noting that
  the singular numbers~$s_{j}(A_{d})$ are the nonincreasing
  rearrangement of the  $d$th tensor power of the univariate numbers~$s_{j}(A_{1})$.
Solving this problem is not a
  trivial task, and the proof~\cite[Thm.~2]{MR3724810}  is based on
  ~\cite[Thm.~4.3]{KSU15}. The singular values for the univariate
  integration operator~$A_{1}$ are known to behave like~$s_{j}(A_{1}) \sim
  \frac{1}{\pi j},\ \text{as}~j\to\infty$.  Therefore, we can
apply~\cite[Thm.~2]{MR3724810} by setting $s=1$ and $c=\pi^{-1}$.

  The weak tractability in $d$ for the multivariate
    integration operator can be
  seen  by comparison with
  the decay rate of~$C(d)$ with
  \begin{equation} \label{eq:KSU1}
 \tilde C(d) \asymp \frac{2^d}{(d-1)!} \quad \mbox{as} \quad d \to \infty.
\end{equation}
The latter has been established in~\cite[Theorem~4.3]{KSU15} (case
$s=1$) for approximation numbers $s_j$ of embedding operators that
follow a rule analog to \eqref{eq:sn-multi}. By virtue
of~\cite[Corollary~5.3]{KSU15} the situation
\eqref{eq:KSU1} represents (quasi-polynomial) tractability,  and
hence weak tractability. Since the decay rate of $C(d)$ from
\eqref{eq:CC} is higher than the rate of $\tilde C(d)$ from
\eqref{eq:KSU1}, this implies that also the direct problem with multivariate integration operators is weakly tractable in $d$.
By virtue of Corollary~\ref{cor:inverse-direct} this also implies
the tractability of the inverse and ill-posed problem.
\end{proof}

\section*{Acknowledgments}
The authors are very grateful to Thomas K\"uhn (Leipzig) for pointing
out the asymptotics in formula~\eqref{eq:CC} for the singular values of the multivariate integration operator.
 The first named author  gratefully acknowledges the support of the Leibniz Center for Informatics,
where several discussions about this research were held during the Dagstuhl Seminar
``Algorithms and Complexity for Continuous Problems'' (Seminar ID 23351).
The second named author has been supported by the German Science Foundation (DFG) under grant~HO~1454/13-1 (Project No.~453804957).

\end{document}